\documentclass[12pt]{article} 
\usepackage{amsmath,amssymb,amsthm,amscd,fancyheadings,float, graphicx, psfrag}

\theoremstyle{plain}

\newtheorem{theorem}{Theorem}[section]
\newtheorem{lemma}[theorem]{Lemma}
\newtheorem{corollary}[theorem]{Corollary}
\newtheorem{remark}[theorem]{Remark}
\numberwithin{equation}{section}
 
\newtheorem{proposition}[theorem]{Proposition}
\newtheorem{example}[theorem]{Example}
\newtheorem{problem}[theorem]{Problem}
\newtheorem{question}[theorem]{Question}

\newtheorem{definition}[theorem]{Definition}
\numberwithin{equation}{section}

\divide\oddsidemargin by 2
        {\end{list}} 

\newsavebox{\savepar}

\newcommand{\id}{{\rm Id}} 
\newcommand{\C}{{\mathbb C}} 
\newcommand{\N}{{\mathbb N}} 
\newcommand{\R}{{\mathbb R}} 
\newcommand{\Z}{{\mathbb Z}} 
 
\newcommand{\Q}{{\mathbb Q}} 

\newcommand{\chat}{{\widehat \C}}

\newcommand{\cale}{{\mathcal E}} 
\newcommand{\calp}{{\mathcal P}} 
\newcommand{\calm}{{\mathcal M}} 
\newcommand{\calk}{{\mathcal K}}

\begin{document}

\title{\bf Dynamics of meromorphic  functions outside a
  countable set of essential singularities} 

\author{P. Dom\'{\i}nguez $^{a}$*, M.A. Montes de Oca$^{a}$ and G. Sienra $^{b}$ \\
{\small $^{a}$ {\it Fac. de Ciencias F\'{\i}sico-Matem\'aticas., BUAP., C.U., Puebla Pue. 72570   M\'exico.}}\\
{\small $^{b}$ {\it Facultad de Ciencias, UNAM.  Av. Universidad 3000, C.U. M\'exico, D.F., 04510, M\'exico.}}}
\date{}

\maketitle
\begin{center}
{\it Dedicated to Professor I.N. Baker}
\end{center}

\maketitle

\begin{abstract} 

We consider  a class of  functions, denoted by $\calk$ in this paper, which are meromorphic outside a compact and countable set $B(f)$, investigated by A. Bolsch in his thesis in 1997. The set $B(f)$ is the closure of isolated essential singularities. 
We review main definitions and  properties of  the Fatou and Julia sets of functions in
class $\calk$. It is studied the role of $B(f)$ in this context. Following Eremenko it is defined escaping sets and we prove some  results related to them. For instance, the dynamics of a function is extended to its singularities using escaping hairs. We give  an example of  an  escaping hair with a wandering singular end point, where  the hair is contained in a wandering domain of $f \in \calk$.


\renewcommand{\thefootnote}{} 
\footnote{{\em The authors were supported by CONACYT  project number 253202.}}  
\addtocounter{footnote}{-2}            
\end{abstract}

 
\section{Introduction}

The iteration of complex analytic functions has its origins in two detailed 
examinations of Newton's method. The first paper  containing two parts (1870 and 1871), \cite{scho1} \cite{scho2}, was written by the German mathematician E. Schr\"{o}der 
(1841-1902) and the second  \cite{cayley} (appeared in 1879) by the British mathematician A. Cayley (1821-1895).

Schr\"{o}der and Cayley each studied the convergence of Newton's method for 
the complex quadratic function. Since then,  many other mathematicians have 
worked  on the problem of analytic iteration and the iteration of complex 
functions, see e.g. \cite{alexander}.
 
In  1907 P. Montel (1876-1975), a French mathematician, received his 
doctorate in Paris which was related to infinite sequences of both real 
and complex  functions. A few years later Montel worked in complex function theory and introduced the theory of normal families. In his papers, published in 
1912 and 1916, he treated Picard's theory. Montel considered his study of 
Picard's theory to be one of  the most important applications of his theory of normal families. Montel's theory of normal families  was quite important in the iteration of 
analytic functions.

Between 1918 and 1920, two French mathematicians, 
P. Fatou (1878-1929) and G. Julia (1893-1978) obtained several results
 related to  the  iteration of rational functions of a single complex variable,
\cite{fatou0}, \cite{julia}. Each of them, based his approach on Montel's
theory of normal families. The main objects of the theory are the Stable set
where the iterates form a normal  family and the Chaotic set which is the
complement of the Stable set. From now on, as it is usual in this area, we will call the Stable set as the
Fatou set and the Chaotic set as the Julia set (of an holomorphic  function $f$) and denote them by $F(f)$  and $J(f)$ respectively.

In 1926 Fatou \cite{fatou} extended some  results of the Fatou and Julia sets
from the rational case to the transcendental entire case. He also  studied the dynamics of the
function $f(z) = e^{-z} +z +1$, for  which the sequence of iteration of points in the Fatou set converges
to $\infty$. Since $\infty$ is an essential singularity of $f$, the component in the Fatou set with such behaviour was called {\it domain of indetermination}. Fatou put forward the conjecture that the Julia set of the exponential map is the Riemann sphere ($J(e^z) = \chat$) and leave
many other  open questions. 

In 1953, H. R$\stackrel{\circ}{\rm a}$dstr\"{o}m  \cite{radstrom} working on
dynamics, showed that among planar domains with  holomorphic self-maps 
$f: D \rightarrow D$, the most dynamically interesting cases are: 
$D = \chat$, if $f$ rational; $D = \C$, if $f$ entire; or 
$D = \C^* = \C \setminus \{ 0 \}$, if $f$ is an endomorphism of the punctured plane. In all other cases every point of 
$D$ is a stable point for $f$.

Some years later,  in 1955,  I.N. Baker used the theory of iteration of analytic
functions, developed mainly by  Fatou and Julia, in his results \cite{Ba},
\cite{baker0} and \cite{baker1a} Baker showed that some dynamical
properties of entire functions are quite different  from those of rational
maps. In 1970 Baker  gave the  first example in  \cite{baker3a}
of an entire function for which the Julia set is the whole Riemann sphere, he
proved that $J(\lambda ze^z) = \chat$ for suitable value of $\lambda$. Later
in \cite{misi}  Misiurewicz proved Fatou's conjecture.   

The connected components in the Fatou set can be either periodic, pre-periodic or wandering, see Section \ref{section3} for definitions. An open problem since Fatou and Julia was the existence of wandering domains. In 1976 Baker (1932-2001) in \cite{bakerw} proved that 
a transcendental entire function with  a multiply connected
Fatou component must be  wandering. In 1985 D. Sullivan \cite{sullivan}  
showed  that for the rational case there were not wandering domains in the
Fatou set, the main  tools in his proof  were Riemann surface theory, planar
topology, prime ends and  the  theory  of quasi-conformal mappings. The last    
concept  was introduced by A. Gr\"{o}tzch \cite{gro}, developed by Teichm\"uller
\cite{teich}, used and  called as we know now by L. Ahlfors in  \cite{ahlfors}. 

There is a  kind  of classification of periodic components in the Fatou set  started by  Cremer  \cite{cremer} and  Fatou  \cite{fatou0}, see for a discussion \cite{bergweiler1}.  Currently,  the classification  due  to I.N. Baker, J. Kotus and L\"u   \cite{baker11}  states  that a periodic component   can be either attracting, parabolic, Siegel disk, Herman ring or indetermination domain.

Using techniques  of quasi conformal surgery, Shishikura in
\cite{shishi} constructed examples of rational functions which have multiply connected rotation domains with the rotation number irrational, these  domains are
Herman rings. For  transcendental entire functions  there are
not Herman rings. The results  mentioned above show that the dynamics of the 
rational  case and the transcendental entire case are quite different.

For the rational  functions  the set of critical values is finite, so with this idea  A. Eremenko and M. Lyubich in \cite{eremenkolyubich} investigated the class of transcendental entire functions, where the set of singular values is finite (asymptotic and critical values). They proved that for functions in their class there are neither  wandering domains nor indetermination domains. The last domains  were called {\it Baker
domains} for the first time  by the authors in  \cite{eremenkolyubich}, thus from now on in this paper we will refer to this name.

Going back to the analytic functions in $\C^*$,  we want to mention some
mathematician who were working on the dynamics of these functions. 
P. Battacharyya  in his PhD thesis \cite{bhattacharyya}, studied the dynamics
of those functions. Later on L. Keen \cite{keen}, J. Kotus \cite{kotus}  and P. Makienko \cite{makienko}  also investigated  the dynamics  of functions in $\C^*$.

We know by definition  that a transcendental meromorphic function  is analytic
in the whole plane except at  poles, each one of those poles goes to infinity
and infinity is either an essential singularity or an accumulation  point of poles.  It was a natural question to ask if for a transcendental meromorphic function, the Fatou and Julia sets could be defined. Between 1990 and 1991  I.N Baker, J. Kotus, and  Y.  L\"{u}   produced  four papers \cite{baker11, baker12, baker13, baker14} in which they investigated the dynamics of transcendental meromorphic functions. For a general survey of transcendental  meromorphic functions, the readers are referred  to \cite{bergweiler1}.

A natural  generalization  is to study the iteration of functions that are 
analytic outside a compact set of singularities, this set is in some 
sense small. A. Bolsch \cite{andreas} and M. Herring \cite{herr} in their PhD
thesis investigated such functions.

Bolsch in  \cite{bolsch, andreas, bolsch1}, investigated the
iteration of analytic functions  outside a compact countable set, which is the closure of isolated essential
singularities. In this  class infinity may  not  be an essential
singularity, so the way to treat this  kind of functions is different to either rational, transcendental entire or transcendental meromorphic functions.

As we can see, the historical development of the study of complex functions 
starts with basic concepts of holomorphic functions on the Riemann sphere 
concerning the Fatou and Julia sets.  However, there are other kind of functions which have been studied extending the theory  of  holomorphic dynamics,  such as
\cite{crowe}, \cite{cruz}, \cite{go}, \cite{jeff} and  \cite{milnor} specially page 17.

In this paper  we want  to continue the research started in \cite{andreas} of a class of functions defined on subsets of the Riemann sphere. 
We can think this class as the smallest semi group with the composition operation,  containing all transcendental meromorphic functions. We denote this class by $\calk$ and we will refer to this kind of functions as $\calk$-meromorphic  functions.

In this  paper the topics are divided as follows: in Section
\ref{section2} we define and study classes of functions, the set of
singularities,  exceptional values and examples of the
definitions mentioned before. Section \ref{dos-2} contains  the iteration of the functions given in Section  \ref{section2} and a classification of their compositions. In Section
\ref{section3} we review the definitions and properties of the Fatou and Julia sets for functions in class
$\calk$ and give some well known  properties for those sets, including the classification of the  periodic components of the
Fatou set. In Section 5 we show some   results related to
the dynamics of functions in class $\calk$. In Section 6
the escaping sets given  by Eremenko  in
\cite{eremenko1} are generalized in two different kinds: $I_e(f)$ which is the escaping set to a specific singularity $e\in B(f)$ and the general escaping set $I_g(f)$ related to the whole $B(f)$. We also classify the dynamics in the general escaping set $I_g(f)$ and extend the dynamics of $f$ to the singularities using curves contained in $I_g(f)$.  Section 7  contains examples of
Baker and wandering domains  for functions in class $\calk$, we use such examples to construct escaping hairs with wandering singular end-points.

 \section{ $\calk$-meromorphic  functions}
\label{section2}

In what follows we shall denote the complex plane by $\C$,  the sphere  
by $\chat$ and the punctured plane by  $\C^* = \chat \setminus \{0,
\infty\}$. We consider the following classes of functions. 
 
\begin{description} 



\item $\cale = \{f: \C \to \C \mid \text{ $f$ is transcendental entire}\}$.

\item $\calp = \{f:\C^* \to \C^* \mid \text{ $f$ is transcendental
    holomorphic}\}$. 

\item We shall distinguish the following two types of  maps in  class $\calp$.
\begin{description} 
\item $\begin{array}{lcl} 
\calp_1 &= &\{ f\in \calp \mid  \text{0 is a pole and an omitted value}\} \\\\
 &=& \{ f(z)=\frac{e^{h(z)}}{z^n}, \text{\  $h(z)$ entire non-constant, $n\in \N$}\}.  
\end{array}$

\item $\begin{array}{lcl} 
\calp_2 &=& \{ f\in \calp \mid  \text{0 is an essential singularity}\}\\\\ 
&=& \{f(z)= z^{n} e^{(g(z)+ h(1/z))}, \text{\ $g$, $h$ entire non-constant, $n\in\Z$}\}. 
\end{array}$  
\end{description} 
\end{description}

\noindent We recall  that a function $f$ is meromorphic  in $\C$ if $f$ is analytic except at  {\it poles}. 
\begin{description}

\item $\calm = \{f: \C \to \chat \mid \text{$f$ is transcendental  meromorphic 
    with one not omitted pole}\}$. 
\end{description}

\noindent A class of functions which contains  all the classes mentioned above and their
iterations is defined as follows.

\begin{description} 
\item  $\calk = \{f: \chat \setminus B(f) \to \chat \mid  f \text{ is  not constant and 
    meromorphic in $\chat \setminus B(f)$} \}$, 
\end{description}

\noindent where the set $B(f)$ is a compact countable set and it is the closure of the set of isolated essential singularities of $f$. We assume that $B(f)$ has at least one  element. In other words, $f\in \calk$ is analytic in 
$\chat  \setminus B(f)$ except at poles. We are calling  this functions {\it $\calk$-meromorphic}.

Examples of  the above classes of functions which for their dynamics have been studied, are  the following:

\noindent (a)  $f_{\lambda}(z) = \lambda e^{z}$ and
$f_{\lambda}(z) = \lambda \sin z$ with   $\lambda \in \C^*$, for functions in class $\cale$; 

\noindent (b)  $f_{\lambda}(z) = \lambda e^{z}/z$, $\lambda \in \C^*$,  and
$f_{\alpha}(z) =e^{\alpha(z - 1/z)}$, where $ 0 < \alpha <  1/2$, for
functions in class $\calp$; 

\noindent (c)  $f_{\lambda}(z) = \lambda \tan z$, 
$f_{\lambda, \mu}(z) = \lambda e^{z} + \mu/z$   and
$f_{\lambda, \epsilon, p}(z) = \lambda \sin z + \frac{\epsilon}{z-p}$ where $p
\in \C$,  and $\mu, \lambda, \epsilon, \in \C^*$, for functions in class $\calm$; 

\noindent (d) $\lambda e^{R(z)} + \mu$ with   
$R(z)$ either a rational function  or a transcendental  meromorphic function,  and $f_{\lambda, \mu}(z) = \tan(\lambda \tan(\mu
z))$, where  $\lambda, \mu \in \C^*$,   for functions in class $\calk$.

\begin{remark}(a) For a function $f$ either in class $\calp, \cale$
 or $\calm$ we know that $\infty$ is an essential singularity or an accumulation of poles, so it is in the
set $B(f)$. 

\noindent (b)  The classes $\calp$, $\cale$ and $\calm$ are disjoint. 
\end{remark}

From now on we will work with functions in class $\calk$,  otherwise we will 
mention the class of functions we are working on. In what follows we will give 
several definitions related to the class $\calk$.

The set of {\it singular values} of $f \in \calk$, denoted by $SV(f)$,  is defined as follows: 

$$SV(f)= \overline{C(f) \cup  A(f)},$$ 

\noindent where $C(f)$ is the set of  critical values of $f$ and $A(f)$ is the
set of asymptotic values of $f$.  We recall that a {\it critical value or an
  algebraic  singularity} is the image of a critical  point. A  point
$a\in\chat$ is called an {\it asymptotic value or a transcendental singularity}
of $f$ at $e$ if there is a path $\gamma(t)\to e$ as $t\to \infty$, such that $f(\gamma(t))\to a$, where $e \in B(f)$. In this case, for every  punctured disc $D_a$ of  $a$, there is an open set $U_a\subset \chat$, such that the restriction $f:U_a\rightarrow D_a$ is an infinite (possibly branched) covering and moreover, there are infinitely many paths $\gamma_i (t)\rightarrow e$, as $t\rightarrow\infty$, such that $f(\gamma_i(t))=f(\gamma(t)))$. We say that $U_a$ is a {\it tract over} $D_a$.

If  $a$  has a simply connected neighbourhood $V$ such that for some component $U$ of the set $f^{-1}(V)$ the map $f: U \to V\setminus \{a \}$ is a universal covering, then  $a$ is called a {\em logarithmic branch point} and $U$ is called an {\em exponential tract}.

We say that $w$ is an {\it omitted value} of $f\in\calk $ if  for every $z \in \chat \setminus B(f)$ we have $f(z) -w = 0$ has no solutions and that $w$  is a {\it Picard exceptional} value of $f$ if $f(z) - w = 0$ has finite solutions. Observe that if $w$ is an omitted value of $f$, then $w$ is a Picard exceptional value of $f$. 

 \begin{definition} 
Let $f\in \calk$  we say that $z_0\in \chat$ has the Picard Property if every
punctured neighbourhood of $z_0$  (in the domain of $f$) is mapped into the Riemann sphere infinitely many times except for at most two points which are   called
local Picard exceptional values of $f$ at $z_0$.
\label{Picard}
\end{definition}

\begin{remark} (a) Let $f\in \calk$. The point $e\in \chat$ has the Picard Property if
  and only if $e\in B(f)$.

(b) If $f \in \calk$  and  $e \in B(f)$, then $e$ is   either an essential isolated singularity or an accumulation of those singularities.

(c) By the Picard Property,  functions in class $\calk$ may have at most two Picard exceptional values. 

(d) A point $w$ is a Picard exceptional value of $f$ if and only if $w$ is a local Picard exceptional value of all singularities. 
\end{remark}

Examples of functions having omitted values, Picard exceptional values and local Picard exceptional values are the following: 

(i) The exponential  map $e^z$ has two
omitted values, zero and $\infty$ which are also Picard exceptional values;

(ii) The map $e^z/z$   has $\infty$ as a Picard exceptional value but it is 
no an omitted value, while zero is an omitted value;

(iii) The map $e^{(z - 1/z)}$, has two omitted
values, $0$ and $\infty$,  such points are also essential singularities.

(iv) If $f(z) = e^{1/z} \tan z$, then $B(f) = \{ 0, \infty\}$ and  the set of
local Picard exceptional values at zero is $\{ 0, \infty\}$. Since $f(k\pi) =
0$  for all $k \in \Z \setminus \{0\}$ and $f((2k-1)\pi /2) = \infty$ for all
$k \in \Z$, therefore,  $0$ and $
\infty$ are not Picard exceptional values of $f$.
\hspace{.5cm}

\begin{problem}   Is there an example of a  function in class $\calk$, such
  that  it has  local Picard exceptional values at each essential singularity
  but not  Picard exceptional values? Observe that such example cannot be  meromorphic.
\end{problem}

\begin{problem} Are the asymptotic values $\pm i$ Picard exceptional values of $f$ in  Example (iv)?
\end{problem}

If  $f$ is a function,  the sequence  formed by its {\it iterates} is denoted
by $f^0:= \id$, $f^n:= f \circ f^{n - 1}$, $n \in \N$, where $\circ$ means
composition and  $f^n$ is the {\em  n-iterate} of $f$.\\

Given $f \in \calk$ we  define the {\it forward orbit, the backward orbit} and
{\it the grand  orbit}  of a point $z$ as $O^+(z) = \{ w: f^n(z) = w$ for  $n
\in {\bf N} \}$, $O^-(z) = \{ w: f^n(w) = z$ for some $n \in {\bf N} \}$ and
$O(z) = O^+(z) \cup O^-(z)$ respectively. Also, we define $E(f)$  as the set
of {\it Fatou exceptional values} of $f$, this is,  points whose inverse
orbit $O^-(z)$ is finite. For instance,  take the map $f(z) = e^z$, the Fatou
exceptional values of $f$  are zero and $\infty$ (both are omitted
values).  Now observe the map $h(z)= - e^z + 1/z$, it   has no omitted values,
the Picard exceptional value is $\infty$ but it is not a Fatou exceptional
value since $f^{-n}(\infty)$ is not finite, so $E(h)=\emptyset$. \\

As a result   of the Iversen's Theorem the corollary in \cite{zhengbook} states the following.

\begin{corollary}
Any Picard exceptional value  of a transcendental  meromorphic function  is
also an asymptotic value. Hence, a transcendental entire function must have
$\infty$ as its asymptotic value
\label{asymtotic}.
\end{corollary}

The above corollary can be extended to functions in class $\calk$, in the sense that every local Picard exceptional value of $f$ is an asymptotic value of $f$. It is clear
that omitted values are Fatou exceptional values which are Picard exceptional
values and by the Corollary  \ref{asymtotic}, they are asymptotic values. These observations are resumed in the diagram in  Figure \ref{diagrama} which  shows the proper contentions between the sets of different kind of values.

\begin{figure}[h] 
\centerline{\includegraphics[height=6cm]{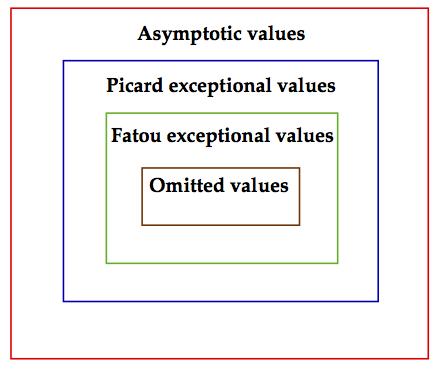}}
\caption{\small Contentions  \label{fig1}}
\label{diagrama}
\end{figure} 

\section{Iteration of functions in class $\calk$}
\label{dos-2}

When we iterate a function $f \in  \cale$ it  can be proved that $ f^n  \in
\cale$ and $B(f^n)=\{\infty\}$ for all $n\in \N$.  Functions in class $\calp_1$ have an omitted pole and one essential singularity,  observe that under the second iteration the omitted pole  becomes an essential  singularity. Thus the new function has two essential singularities. For instance, the map $f(z)=  e^z/z$  has $B(f)=\{\infty\}$ and $B(f^n)=\{\infty,0\}$ for all $n\geq 2$. For $f \in \calm$ is not true that  $f^n  \in  \calm$, in this case, $B(f)=\{\infty\}$ but, due to the existence of poles, $B(f^n)$ has more than one point for all 
$n\geq 2$.

Bolsch in \cite{andreas} proved that for functions $f$, $g$ in class $\calk$  the composition 
$f \circ g$ is in $\calk$, where $B(f\circ g) = B(g) \bigcup g^{-1}(B(f))$. For $n \in \N$ we denote
$f^{-n}(B(f)) = \{z: f^n(z) = e \in B(f) \}$.

The following theorem follows from \cite{bolsch} and \cite{herr}.

\begin{theorem}
 If $f \in \calk$, then
  $f^n\in\calk$ and for each  $n \geq 1$ the natural boundary of $f^n$ is the set  
  
$$B_n = B(f^n) = \bigcup_{j=0}^{n-1} f^{-j}(B(f)),$$

so that $f^n$ is meromorphic function in the region 

$$D_n = \chat \setminus B_n.$$ 

Further, the sets  $B_n$ are all  compact and countable. Moreover,  $f^n$ cannot be continued meromorphically over any point of $B_n$.
\end{theorem}

We say that $p$  is a {\it pre-pole of order $n$} in case that  $p \in  f^{-n}(\infty)$, for some $n \geq 2$.\\

We say that  $p$  is a {\it pre-singularity of order $n$} if $e \in B(f)$ and $p \in  f^{-n}(e)$,  for some $n \geq 1$. Observe that for the function $-e^z+1/z$ zero is a pole and a pre-singularity of order one, while the function $e^{1/z}+1/(z-1)$ has one as pole and it is not a pre-singularity of any order.

From now on, we will denote the set of all pre-singularities and singularities  of $f \in \calk$ by  $B^{-}(f) = \bigcup_{j=0}^{\infty} f^{-j}(B(f))$. Observe that   the set $B^{-}(f)$ is countable. \\

In what follows we will  define the following classes of functions.

\begin{description} 

\item $\calk_0 = \{ f \in \calk \ | \; \# B(f) = 0 \}$;

\item $\calk_n = \{ f \in \calk| \; \# B(f) = n, n \in \N\}$;

\item $\calk_{\infty} = \{ f \in \calk|  \; \# B(f) = \infty \}$.

\end{description}

The symbol $\# B(f)$ means the number  of isolated essential singularities of
the function $f\in \calk$.

\begin{remark} (a) Functions in classes $\cale$,  $\calp_1$  and $\calm$  are in $\calk_1 = \{ f \in \calk| \; \# B(f) = 1 \}$, for such functions $B(f)= \{\infty \}$. If $e \in B(f)$ it can be normalized to be $\infty$, in this case $\calk_1 = \cale \cup \calp_1 \cup \calm$.

\noindent (b) Functions in  class $\calp_2$ are in $\calk_2 = \{ f \in \calk| \; \# B(f)
= 2 \}$, where $B(f)= \{\infty, 0\}$. If $f \in \calk_2$, $B(f) = \{e_1, e_2
\}$ and $e_1, e_2$ are omitted values, so   they can be normalized to be $\{\infty, 0\}$.
\end{remark}

\begin{remark} 

\noindent(a) If $f\in \cale$, then $B^-(f)=\{\infty\}$.

\noindent (b) If $f \in \calp_1$, then $B^-(f)=\{\infty,0\}$ ($0$ is an omitted pole of $f$).

\noindent (c) If $f \in
  \calp_2$, then $B^-(f)=\{\infty,0\}$ ($0\in B(f)$ and $0$ is an omitted value). 

\noindent (d) After normalization we have that, if $f\in \calk \setminus (\cale\cup \calp)$, then $B^-(f)$ is an infinite set.
\end{remark}

For functions $f \in \calk_n$ and $g \in \calk_m$, where $n,m \in \N \cup \{0, \infty \}$, the composition 
 $f \circ g$  (or $g \circ f$) is in  $\calk_r$ for some $r \in \N \cup \{0, \infty \}$. To calculate $r$ we should consider several cases, for instance the ones given below.\\

 Case I. $f \in \calk_0$ and $g \in \calk_1$. 
 
 (a) The composition $f \circ g$ is in $\calk_1$.  We recall that $\calk_1$ is the disjoint union of $\cale, \calp$ and $\calm$, thus the  the composition of  $f \circ g$ can be in any of them.  For instance: take  $f(z) = z^n + 1$  a polynomial in $\calk_0$ and $g(z) = e^z/(z - 1)$  a function in class $\calm$.  The composition $(f \circ g)(z)  =  e^{nz}/(z - 1)^n  + 1$, thus $(f \circ g)(z)  \in \calp_1 \subset   \calk_1$. 
 
(b) Now, the composition  $g \circ f$ is in $\calk_r$, where $r$ is the number of poles of $f$ without multiplicity. In the special case that $f$ is a polynomial, $f(z)=\infty$ implies that $z=\infty$ and so $g \circ f \in \calk_1$.  In contrast with the previous example, if $f(z)=z^2/(z-1)\in \calk_0$ and $g(z)=e^z$, then $(g\circ f)(z)=e^{z^2/(z-1)}\in \calk_2$ and $B(g\circ f)=\{\infty,1\}$, while $f\circ g\in \calk_1$. Observe that the singularity $1$ is not omitted by $g\circ f$, so after normalization, $g\circ f \in\calk_2\setminus\calp_2$.

Case II. $f \in \calk_0$ and $g \in \calk_n$.

 The composition $f
\circ g$ is in  $\calk_n$. For instance, given $g\in\calk_2$, the composition $f
\circ g$ is in $\calp_2$ if and only if $f(z) = \lambda z^m$ for some $m \in
\Z \setminus \{ 0\}$ and $\lambda \in \C^*$. We recall that $B(g)= \{\infty,
0\}$ after normalization.\\

The following proposition does not need a proof since it is clear by the
definitions and the basic concepts presented above. 

\begin{proposition}  After normalization, the following cases hold:

I. For  $f \in K_1$ and  $B(f)= \{ \infty \}$ we have:

(i) If  $f \in \cale$, then  $B(f^n)= \{ \infty \}$, for $n \in \N$.

(ii) If  $f \in \calp_1$, then  $B(f^n)= \{ \infty, 0 \}$, for  $n \geq 2$.

(iii) If  $f \in \calm$, then  $\# B(f^n)= \infty$, for $n \geq 3$.\\

II. For $f \in K_2$ and $B(f)= \{ \infty, 0 \}$ we have:

(i) If  $f \in \calp_2$, then  $B(f^n)= \{ \infty, 0 \}$, for  $n \in \N$.

(ii) If  $f \in K_2 \setminus \calp_2$, then  $\# B(f^n)= \infty$, for  $n \geq 3$.\\

 III. For  $f \in K_m$ and  $m \geq 3$,  the number of essential singularities is  $\# B(f^n)= \infty$, for all  $n \geq 2$. \\
 
IV. For $f\in\calk \setminus \{\ \cale,\calp\}$,  $f^n \in K_{\infty}$ for all $n \geq 3$.
\end{proposition}

\section{The Fatou and Julia sets }
\label{section3}

In this section we will work with the dynamics of functions in class $\calk$, having in mind the properties of the Julia and Fatou sets in classes $\cale$, $\calp$ and $\calm$. We will start the section with the following concepts.

 Let $f \in \calk$. If  $n$ is the minimum such that $z$ satisfies $f^n(z) = z$, we say that $z$ is a
 {\em periodic point of period $n$}.  The set $\{z = z_1 , z_i= f(z_{i-1}),... z_{n} = z_1 \}$, $i = 2, 3,..., n$, is called {\it a periodic cycle} at $z$. If $z$ is a fixed point of period  $n$ the quantity
$$ 
\prod_{j=1}^{n} f^{'}(z_i) = \lambda(z)
$$
is called the {\it eigenvalue or multiplier} of the cycle at $z$. When any value $z_i$ is
the point at infinity the factor $f^{'}(z_i)$ is replace  by the derivative of 
$1/f(1/z)$ at the origin.

The classification of a periodic point $z_0$ of period $p$ of $f\in\calk$ is
given as follows:
 
\noindent (a) {\it super-attracting} if $|\lambda (z_0)|=0$; 

\noindent (b) {\it attracting} if $0<|\lambda(z_0)|<1$; 

\noindent (c) {\it repelling} if $|\lambda(z_0)|>1$; 

\noindent (d) {\it rationally indifferent} if $|\lambda(z_0)|=1$ and
$(f^n)'(z_0)$ is a root of unit, in this case $z_0$ is known also as a {\it parabolic periodic point}; 
 
\noindent (e) {\it irrationally indifferent} if $|\lambda(z_0)|=1$, but $\lambda(z_0)$ is not a root of unit.\\

Now we are able to give the definitions of the Fatou and Julia sets and  some
of their properties.\\

For functions in class $\calk$ we define  the {\em Fatou set}, denoted by
$F(f)$, as the maximal open set $G$  such that all $f^n$ are analytic and
 forms a normal family in $G$ (in the sense of Montel). The complement of
 the Fatou set  is called the {\em Julia set} which is  denoted by $J(f) = \chat \setminus F(f)$ and therefore $B(f)\subset J(f)$.

The following  properties  of the Fatou and Julia sets were initially proved
by Fatou and Julia, in \cite{fatou0} and  \cite{julia}  respectively, for
rational functions. Fatou \cite{fatou}  and Baker \cite{Ba}
proved the properties for transcendental entire functions $\cale$ and Battacharyya
\cite{bhattacharyya} for functions in $\calp$. While Baker, Kotus and  L\"u  \cite{baker11} for transcendental meromorphic functions $\calk_1$ and  Bolsch \cite{andreas},  for  functions in class $\calk$. We
mention that the properties of  the Julia  and  Fatou sets  are much the same
for each of those classes, but different proofs were needed and some discrepancies arise  by the time they were proved.  The properties are the following.\\

\noindent (a) The Fatou set $F(f)$ is open and the Julia set $J(f)$ is  closed. 

\noindent (b) The Julia set $J(f)$ is perfect.

\noindent (c) The set $F(f)$ is completely invariant, this means, it is forward invariant $f(F(f))\subseteq F(f)$ and backward invariant $f^{-1}(F(f))\subseteq F(f)$. $J(F)$ is also completely invariant in the sense that $f(J(f)\setminus B(f))\subseteq J(f)$ and $f^{-1}(J(f))\subseteq J(f)$.


\noindent (d) For a positive integer $p$, $F(f^p) = F(f)$ and $J(f^p) = J(f)$.

\noindent (e) The Julia set $J(f)$ is the closure of the set of repelling
periodic points of all periods of $f$.\\

For a function in class $\calk$  a Fatou component $U$ can be either:

(a) {\it periodic} if $f^n(U) \subset U$, for some $n \geq 1$; 

(b) {\it pre-periodic} if $f^m(U)$ is periodic for some integer $m \geq 0$ or

(c) {\it wandering} if $U$ is neither periodic nor pre-periodic.  \\

Consider $U_1\subseteq F(f)$ a periodic component, we say that $\{U_i\}$, $1\leq i \leq p$, is a \textit{periodic cycle of period p} if $f(U_i)\subseteq U_{i+1}$, for all $1\leq i\leq p-1$, and $f(U_p)\subseteq U_1$.

If $U$ is a periodic component of $F(f)$ of period $p$, the classification of the periodic component is given as follows. 

\begin{enumerate}
  \item $U$ contains an attracting periodic point $z_{0}$ of period $p$
    such that  $f^{np}(z)\rightarrow z_{0}$ for $z\in U$ as
    $n\rightarrow\infty$, then $U$ is called the {\it immediate attracting
      component} of $z_{0}$.

  \item $\partial U$ contains a periodic point $z_{0}$ of period $p$ and
    $f^{np}(z)\rightarrow z_{0}$ for $z\in U$ as $n\rightarrow\infty$. Then
    $(f^{p})^{'}(z_{0})=1$ if $z_{0}\in\C$.  
    In this case,
    $U$ is called either a Leau domain or a {\it parabolic component}.

  \item There exists an analytic homeomorphism $\varphi: U\rightarrow D$ where
    $D$ is the unit disc such that $\varphi(f^{p}(\varphi^{-1}(z)))=e^{2\pi i
      \alpha} z$ for some $\alpha\in \R\setminus \Q$. The component $U$ is
    called a {\it Siegel disc}.

  \item  There exists an analytic homeomorphism $\varphi: U\rightarrow A$
    where $A$ is an annulus $A=\{z: 1< \left|z\right|< r\}$, $r>1$, such that
    $\varphi(f^{p}(\varphi^{-1}(z)))=e^{2\pi i \alpha} z$ for some $\alpha\in
    \R\setminus \Q$. The component $U$ is called a {\it Herman ring}.

  \item There exists $z_{0}\in\partial U$ such that $f^{np}(z)\rightarrow
    z_{0}$, for $z\in U$ as $n\rightarrow\infty$, but $f^{p}(z_{0})$ is not
    defined. In this case, $U$ is called a {\it Baker domain} and  the 
point $z_{0}$ is called the {\it absorbing point} of $U$.
\end{enumerate}

We recall that $SV(f)$ is the set of singular values of $f \in \calk$. We define  the {\em postcritical set} of $f$  as follows:

$$SV^+(f)=\bigcup_{n\in\N} f^n(SV(f)).$$

The relation between some of the periodic Fatou components of a function $f \in \calk$ and the set of its singular values $SV(f)$ is given below in Theorem \ref{relation}, see \cite{herr} for a proof.  

\begin{theorem} \label{relation}
Let $f\in\calk$ and $\{U_i\}$ a periodic cycle of period $p$.

(a) If the elements of the cycle are attracting or a parabolic components, then for some $i\in \{1,...,p\}$, $U_i\cap SV(f)\neq \emptyset$.

(b) If $U_i$ is a Siegel disk or a Herman ring, then $\partial U_i \subset \overline{SV^+(f)}$.
\label{relation}
\end{theorem}

 The relation between $SV(f)$ and a Baker domains is more complicated as the following cases show, for each one of them, there exist a transcendental meromorphic function $f$ with an invariant Baker domain $U$ satisfying: (a) $U\cap SV(f)\neq \emptyset$; (b) $U\cap SV(f)=\emptyset$; (c) $\partial U\subset \overline{SV^+(f)}$, and (d) $\partial U$ has no finite points of $\overline{SV^+(f)}$. For more details see \cite{bergweiler1}.\\
 
 The following theorem of Baker given  in \cite{baker2002}, for a more general class of functions, shows the relation between wandering domains of $f\in\calk$ and $SV(f)$.
 
 \begin{theorem}
 If $f\in \calk$ and $U$ is a wandering domain of $f$, then any limit function of a sequence in $U$ is a constant which is an accumulation point of $SV^+(f)$.
\end{theorem}  
 
 


 Concerning Herman rings we shall mention some results  for the different classes of functions given in Section \ref{section2}.

(a) For functions in classes $\cale$ and $\calp_1$ there are no Herman rings. It is possible  to have doubly connected domains in $F(f)$ for $f \in \cale$ which are wandering domains, see the discussion in Subsection \ref{wandering}.

(b) For functions in   class $\calp_2$  it is known that the Fatou set can have at most one doubly  connected component in the Fatou set. An example given in \cite{baker2a} of a  function with a  Herman ring  is $\lambda z e^{\alpha( z - 1/z)}$ with $0 < \alpha <1/2$, $\lambda= e^{2\pi i \beta}$ and $\beta$ chosen in such a way that satisfies  a diophantine condition.

(c) For functions in class $\calm$, it was shown in \cite{fagellapat}, 
that it is
possible to produce Herman rings.  Even more,  they constructed  Herman rings  for some
$f \in \calk$. \\

In the following two subsections  we shall  give a discussion related  with some examples of Baker  domains and wandering domains for the  different classes of functions  given in Section \ref{section2}.

\subsection{Baker domains}
\label{baker}

Let $\{U_i\}$ be a periodic cycle of period $p$, we say that it is a \textit{cycle of Baker domains} if for all $1\leq i\leq p$, $U_i$ is a Baker domain. \\


If $f \in \calk$, $U$ is   a Baker domain  and $z_0$ is an absorbing point of $U$, then  we have the following statements. 

(i) $z_0$ is in $B^-(f)$ and for any cycle of Baker domains at least one absorbing point is in $B(f)$.

(ii) If $z_0 \in B(f)$ and $w_0$ is the absorbing point  of a Baker domain $V$ such that $f(U) \subset V$, then $w_0$ is an asymptotic value of $f$ at $z_0$.\\

\noindent {\bf I. Baker domains for functions in class $\cale$.}\\

 Having in mind the classification of a  periodic component $U$  in the Fatou
set, it  is natural to study $U$ and its boundary in connection
with the Riemann map $\Psi: D(0, 1) \to U$, were $D(0,1)$ is the disk with center at 0 and radius 1.  We  denote by $\Psi(e^{i\theta})
= \lim_{r \to 1} \Psi(re^{i\theta})$, $e^{i\theta} \in \partial D$,  the {\it radial limit} and define 
$$
\Theta =  \{ e^{i\theta}: \Psi(e^{i\theta}) = \infty \}.
$$

For $e^{i\theta} \in \partial D$ and $g$ analytic in $D$ the {\it cluster set}
  $C(g, e^{i\theta})$ is the set of all $w \in \chat$  for which there exist
  sequences $z_n$ in $D$ such hat $z_n \to e^{i\theta}$ and $g(z_n) \to w$ as
  $n \to \infty$. Now define  the set 
$$
\Xi =  \{ e^{i\theta}: \infty \in C(g, e^{i\theta}) \}.
$$

Observe that $\overline{\Theta} \subset \Xi$. In \cite{bakerwe} Baker and Weinreich  proved the following theorem.

\begin{theorem}
If $f \in \cale$ and $U \subset F(f)$ is an unbounded invariant component
which   is not a Baker domain, then  $\infty \in C(\Psi,e^{i\theta})$, for
every $e^{i \theta} \in \partial D$, this is $\Xi = \partial D$.
\label{wei} 
\end{theorem}

Theorem \ref{wei}  no longer holds  for Baker domains. An
example was given in \cite{bakerwe} where $f(z) = z + \gamma + e^{2\pi i}$, for
suitable  choices of $\gamma$, so that  $f^n \to \infty$ in $U$ and
$\partial U$ is a Jordan curve in $\chat$, so each $C(f, e^{i\theta})$ is a
different singleton, i.e., $\Theta$ and $\Xi$  consist of only one point.
In \cite{ber} Bergweiler showed that $h(z) = 2 - log2 + 2z - e^z$ has the same property.\\

Baker and Wienreich also proved that if $ U$ is  a Baker domain  and $\partial
U$ is a Jordan curve in $\chat$, then $f$ is univalent in $U$.

The function $f(z) = e^{-z} + z + 1$  is one of the functions discussed in
the  Fatou's fundamental paper \cite{fatou},  for which $f^n \to \infty$ in
$U$, where $U$ contains the right half plane, in other words, $f$ has a Baker domain. Now if  we take the function $G(z) = f(z) + \epsilon + k(z)$, where
$f(z)= e^{-z} + z + 1$ (Fatou's function),    $\epsilon \geq 0$ is a constant and
$k$ is an entire function such that $|k(z)| \leq Min(\epsilon, 1/|z|^2)$ on the
strip $S = \{z = x + iy: |y| < \pi, x <0 \}$. Then $G(z)$ has a Baker domain
$U$ and $\overline{\Theta} = \partial D$, see  \cite{pato} for a proof. When $\epsilon = 0$ we get $G(z) = f(z)$, so we have the Fatou's
function. 

The following theorem was proved in \cite{pato} and \cite{kisaka}.

\begin{theorem}
If $f \in \cale$ and $U$ is an unbounded invariant component, which is a Baker
domain, such that $f|U$ is not univalent, then $\overline{\Theta}$
contains a non-empty perfect set in $\partial D$.
\label{pati}
\end{theorem}

With the hypothesis of Theorem \ref{pati}, is it necessarily the case that
$\overline{\Theta} = \partial D$?  Bargman in \cite{bargman} proved that this
is the case  for a doubly parabolic Baker domain that we define below.\\

Classifications  of a Baker domain have been given  by Cowen in \cite{cowen},  by N. Fagella and C. Henriksen  in \cite{fagelac}. Here we mention the one given in \cite{fagelac}.\\

Let $f \in \cale$ and $U$ be an invariant Baker domain of $f$. Then  $f|U$ is a
Riemann surface equivalent to exactly on of the following cylinders:

(a) $\{ z \in \C: -s < Im z < s \}/ \Z$, for some $s >0$;

(b) $\{ z \in \C: Im z > 0 \}/ \Z$;

(c) $\C/\Z$.

In case (a) $U$ is called {\em hyperbolic}, in case (b) {\em simply parabolic}
and in case (c) {\em doubly parabolic}.\\

Examples of the above classification of invariant Baker domains when: (i) the function is univalent can be found, for instance in \cite{bakerwe},
\cite{baranskif}, \cite{ber}, \cite{eremenkolyubich}  \cite{fagelac} and \cite{herman}. (ii) For hyperbolic Baker domains in which the function is not univalent we refer to \cite{rippon30} and \cite{rippon}.  

Bergweiler and Zheng  point it up  that  the classification  given above for an
invariant Baker domain  is the same as the given by Cowen, see the discussion
in Section 2 in \cite{berzhen}. Also they  showed  the following two theorems.

\begin{theorem}
There exist $f \in \cale$ with a simply parabolic Baker domain in which
 $f$ is not univalent.
\label{theow}
\end{theorem}

The Baker domain in Theorem \ref{theow} can be chosen such that $\Xi \neq \partial D$. In particular this implies that $\overline{\Theta}\neq \partial
D$. The following Theorem   gives the existence of an example of a hyperbolic
Baker domain  with this property.

\begin{theorem}
There exists $f_1 \in \cale$ with a simply parabolic Baker domain $U_1$ such
that $f_1|U_1$ is not univalent and the set $\Xi$ defined above satisfies $\Xi
\neq \partial D$.
There also exists $f_2 \in \cale$ with hyperbolic Baker domain $U_2$, satisfying
$\Xi \neq \partial D$ such that  $f_2|U_2$ is not univalent.
\end{theorem}

An interesting  family studied by Rippon and Stallard in
\cite{rippon0} and \cite{rippon} is $f(z)=az+bz^ke^{-z}(1+o(1))$ where $k\in
\N$, $a\geq 1$ and $b>0$. They proved that $f$ has an invariant Baker domain
$U$ where $f$ is not univalent. Also,  they proved the following result.

\begin{theorem}
For each $n\in \N$, there exist an entire function $f$ which has a periodic
cycle of period $n$ of Baker domains on which $f$ is univalent.
\end{theorem}

Another transcendental entire family  with Baker domains  studied by Lauber in \cite{lauber} is the  family $f_c(z) = z - c + e^z$. More precisely  he proved: (a) If Real$(c) >0$, then the Fatou set of $f_c$ consists of only one component which is a Baker domain and
(b) If $c \in i\R$ there are two possibilities to have Baker domains:
(i) If $c = 0$, then the Fatou set consists of infinitely many Baker domains
and their pre-images.
(ii) If $c = 2\pi i \alpha$, where $\alpha$ is in the Brjuno set, then the
Fatou set of $f_c$ consists of a univalent Baker domain and its pre-images.
(iii) If $c = 2\pi i \alpha$, where $\alpha$ is irrational not in the Brjuno
set, then the Fatou set  of $f_c$ is empty. Accordingly to the previous cases, we can observe that the Baker domain  can vanish or
split into infinitely many components as Real$(c)$ goes to zero.\\

\noindent {\bf II. Baker domains for functions in class $\calp$.}\\

Mart\'{\i}-Pete in \cite{martitesis} gave the first explicit examples of functions in class $\calp$ with Baker domains which can be hyperbolic, simply parabolic or doubly parabolic. For instance,  the function $f(z)=\lambda ze^{(e^{-z}+1/z)}$,  for every $\lambda>1$, has a hyperbolic Baker domain escaping to $\infty$, and the function  $f(z)=ze^{(e^{-z}+1)/z}$ has a  simply parabolic Baker domain escaping to $\infty$; observe that in both  examples the absorbing point is $\infty$.\\

\noindent {\bf III. Baker domains for functions in class $\calm$.}\\

If $f \in \cale$ and $U$ is an unbounded component in the Fatou set it is
known that $U$ has to be simply-connected. This  result not  longer hold if
$f$ is allowed to have even one pole. For instance,  the function $f(z) = z + 2 +e^{-z} + \frac{\epsilon}{z -a}$,  where $\epsilon = 10^{-2}$ and $a = 1 + \pi
i$,  has a multiply-connected unbounded invariant component $H$ in the Fatou
set in which $f^n \to \infty$, see \cite{pat} for a proof.\\

The function $f(z) = -e^z + 1/z$ given in \cite{baker13}  has a Baker domain of
period two, this is, the function  has a cycle of Baker domains  $\{F_{\infty},F_0\}$ say, for which $\infty$ and zero are the absorbing points of $F_{\infty}$ and $F_0$ respectively. 

A generalization of the above example is the  family $f(z) =  \lambda e^z + \mu/z$, $\lambda, \mu \in \C^*$, for some parameters $\lambda, \mu$,  which has been  studied in \cite{marcos}. In this family, some of the Baker domains have infinitely many critical points. In \cite{marcotesis} some interesting problems concerning the connectivity of the Baker domains in this family were presented.\\
 
\noindent {\bf IV. Baker domains for functions in class $\calk$.}\\

The first example of a function in $\calk_2 \setminus \calp_2$ with Baker domains  was given in \cite{baker13},  as we mentioned before. The Baker domains $F_0$ and $F_\infty$ of $f(z)=-e^z +1/z$ are invariant for $f^2\in\calk_2 \setminus \calp_2$. This remain true for the functions in \cite{marcos}, in this case, note that zero has infinitely many preimages, then $f^n\in \calk_{\infty}$ for $n\geq 3$. Even more, when $n$ is even $f^n$ has two invariant Baker domains and if $n$ is odd then $f^n$ has a Baker domain of period two.

Another explicit example of a function in class $\calk$ with Baker domains is treated   in Section \ref{section6} of this paper. Independently of this work,  we know by  A. Esparza (oral communication) that using results  of Rippon and Stallard  there are functions  in class  $\calk$ of the form $f(z) = z + e^{g(z)}$, where $g(z)$ is meromorphic , having  Baker domains.

\subsection{Wandering domains}
\label{wandering}

We recall that a wandering domain is a domain which is neither periodic nor pre-periodic.
In what follows we mention some examples of wandering domain given by different authors.\\

\noindent {\bf I. Wandering  domains for functions in class $\cale$.}\\
 
 For functions in class $ \cale$  Baker in \cite{bakerw} constructed  the first example of a multiply connected wandering domain, in such example the connectivity is infinite.  Latter on examples of simply connected wandering components,  for functions in class $\cale$,   were  constructed  by Eremenko and Lyubich by using results on approximation theory, see \cite{eremenkolyubich}.  Also,   Herman  in \cite{herman1} gave  the  examples  $g(z) = z + 1 + e^{-z} + 2\pi i$ and $h(z) = z + \lambda \sin (2\pi z) + 1$, where $1 +2\pi \lambda = e^{2\pi i \alpha}$ for suitable $\alpha \in \R$ which have   simply  connected wandering components.  Since the above examples,  several mathematicians have discussed the existence of wandering components  see \cite{bergweiler1} for more examples.
 
 An open question,  posted by Baker,  of whether multiply connected  wandering components of  finite connectivity could be possible. A positive  answer was given  by Kisaka and  Shishikura who  constructed a wandering domain of connectivity two  by using 
 quasi-conformal surgery, see \cite{kisaka2} for details.\\

\noindent {\bf II. Wandering  domains for functions in class $\calp$.}\\

Baker \cite{baker2a}, Kotus  \cite{kotus2}  and Mukhamedshin  \cite{mu} showed that there are functions in class $ \calp$ having wandering  domains  by using approximation theory. Mart\'{\i}-Pete gave in \cite{martitesis} an  explicit example, $f(z)=ze^{(\sin z + 2\pi)/z}$ which has a bounded wandering domain $U$, the orbit of points in 
$U$ tend to infinity.\\

\noindent {\bf III. Wandering domains for functions in class $\calm$.}\\

In \cite{baker12} the authors presented, using  results on complex approximation,
examples of wandering domains  (bounded or unbounded) of any connectivity, so
they solved the connectivity problem for transcendental meromorphic
functions.\\

\noindent {\bf  IV. Wandering domains for functions in class $\calk$.}\\

Doubly connected wandering domains  for funtions in  class $\calk$ have been given in \cite{memopat}, by using quasi-conformal surgery.  In Section \ref{section6},  of this paper,   we give an example of a function $f \in \calk$  which has   infinity  wandering domains with disjoint orbits.

\section{Some results related to functions in class $\calk$}
\label{section4}

For a periodic component in the Fatou set, Bolsch in \cite{andreas} prove the
following theorem by using Ahlfors theory of covering surfaces.

\begin{theorem}
If $f \in \calk$  and $U$ is a periodic component of the Fatou set,  then $U$
has connectivity  1, 2 or $\infty$.
\label{uno} 
\end{theorem}  

The above  result improves the result of Baker, Kotus and L{\"u}  given in \cite{baker13} for invariant components.



Concerning completely
invariant components in the Fatou set the  Lemmas 4.1, 4.2 and 4.3  in \cite{baker13} can be generalize to functions in class $\calk$.

\begin{lemma}
Let $f \in \calk$.  A completely invariant component of the Fatou
set $U$  has connectivity either $1$ or $\infty$.  
\label{dos2} 
\end{lemma}

\begin{lemma}
Let $f \in \calk$ and $U$ be a completely invariant component of the Fatou
set.  Then $\partial U = J(f)$.
\label{tres} 
\end{lemma}

\begin{lemma}
Let $f \in \calk$. If there are two or more completely invariant components in
the Fatou set, then each one is simply connected.
\label{cuatro}
\end{lemma}
 
As an analogy with the definition given by Eremenko and Lyubich in
\cite{eremenko} we define the  class $\mathcal{S}_{\calk}$  as the set of
functions $f \in \calk$ such that the set of singular values is finite. We denote the sets 
$\mathcal{S}_{\cale}$, $\mathcal{S}_{\calp}$ and $\mathcal{S}_{\calm}$ for the classes of functions
$\cale$, $\calp$ and $\calm$ respectively.\\

For functions in class $\mathcal{S}_{\calm}$ it was proved by Baker, Kotus amd L\"u \cite{baker13} that the Fatou set has at most  most two completely invariant domains. 
It is an open question   $f \in \calm$ have at most two completely invariant domains. The same question  is open for functions in class $\calk$.

In \cite{baker14} and  in \cite{eremenko} the authors proved that  for
functions   in class  $\cale$ and  $\calm$,   with a finite set
of singular values there are neither  wandering  components nor Baker
domains. The same statement works for functions in class $\calk$ and  the
proofs are similar to those in \cite{baker14} or in \cite{adam1},  so we state the following Theorem.

\begin{theorem}
For functions in class $\mathcal{S}_{\calk}$ there are neither wandering
domains nor Baker domains.
\label{cinco}
\end{theorem}

\begin{theorem} \label{julianoespolvo}
 Let $f\in S_{\calk}$ and suppose that there is an attracting fixed
 point whose Fatou component $H$ contains all the singular values of
 $f$. Then $J(f)$ is totally disconnected.
\label{seis}
\end{theorem}

For functions in class $\cale$ and class $\calp$ we must recall that $\infty$  is an exceptional Picard value, which is in the  Julia set and it is a singular value of $f$.  Thus functions in class $\cale$ and $\calp$ do not satisfy the conditions of Theorem \ref{julianoespolvo}.\\

The following result is a corollary of Theorem 9.1.1 in \cite{herr} which applies to functions in class $\calk$.

\begin{theorem}
Let $f\in \calk$. If $U$ is a component of $F(f)$such that $B(f)\cap \partial U=\emptyset$, then $f:U\rightarrow f(U)$ is a finite branched cover.

\end{theorem}

\begin{proposition}
Let  $f \in \calk$,  $U$ be a Fatou multiply connected component, $\gamma
\subset U$  be a simple closed curve and $\Delta$  be  the compact region
bounded by $\gamma$  such that $\Delta\cap B(f)=\emptyset$, $\Delta$ has no singular values of $f$ and $Int(\Delta) \cap J(f) \neq \emptyset$. Then every $V \subset f^{-1}(U)$ a
connected component is multiply connected.  
\label{seis2}
\end{proposition}

{\bf Sketch  of the proof of Proposition \ref{seis2}.}  Let $V\subset f^{-1}(U)$ a connected component. As $f$ is a covering on $V$ to $U$ (possibly branched), then there is $D\subset f^{-1}(\Delta)$ a connected component such that $D\cap V\neq \emptyset$. Since $f$ on $D$ is a covering (no branched) to $\Delta$, then $Int(D)$ is a simply connected domain such that $Int(D)\cap J(f)\neq \emptyset$, $D\cap B(f)=\emptyset$ and $f(\partial D)=\gamma$, therefore $\partial D\subset F(f)$, so $\partial D\subset V$, thus $V$ is multiply connected.

\begin{corollary}
If  $f \in S_{\calk}$, $U$ is a Fatou component  of infinity connectivity,
 and   $V \subset f^{-1}(U)$ is a  connected component, then $V$ is of infinity 
connectivity.
\label{seis1}
\end{corollary}
 
\begin{proposition}
Let  $f \in \calk$,  $U$ be a Fatou multiply connected component, $\gamma
\subset U$  be a simple closed curve and $\Delta$  be  the compact region
bounded by $\gamma$  such that $Int(\Delta) \cap J(f) \neq \emptyset$, $\Delta
\cap B(f) = \emptyset$ and $\Delta$ has no poles. If $V$ is the connected component in $F(f)$ containing $f(U)$, then $V$ is multiply connected.  
\label{seis3}
\end{proposition}

{\bf Sketch  of the proof of Proposition \ref{seis3}.} Since $f$ is analytic on $\Delta$, then $f(\Delta)$ is a compact subset of the plane, by the maximum principle the interior of $\Delta$ is mapped in the interior of $f(\Delta)$, therefore $f(\Delta)\cap J(f)\neq\emptyset$, so $V$ is multiply connected.

\begin{corollary}
Let  $f \in \calk$,  $U$ is a Fatou component  of infinity connectivity,   $B(f)$ is  finite and $f$ has finite poles.  If $V$ is the connected component in $F(f)$ containing $f(U)$, then $V$ is multiply connected.  
\label{seis4}
\end{corollary}

\begin{corollary} \label{hoyosvanahoyos}
 Let  $f \in (\cale \cup \calp \cup \calm)$ with a finite set of poles  and
 $U$ is  a Fatou component  of infinity connectivity. If $V$ is the connected component in $F(f)$ containing $f(U)$, then $V$ is multiply connected.
\label{seis5}
\end{corollary}

\begin{remark}
In Corollary \ref{hoyosvanahoyos} there is an example in \cite{rippon2008} where the
connectivity of $U$ is infinite but the connectivity of $V$ is finite.
\end{remark}

\section{The escaping set for functions in class $\calk$}
\label{section5}

In 1989 Eremenko in \cite{eremenko1}  defined the escaping set $I(f)$ of a
transcendental entire function and proved some important properties of $I(f)$. 
Many researchers have been studied the escaping set since then, for instance Bergweiler \cite{berg}, Dom\'{i}nguez  \cite{pat}, Osborne \cite{osborne},  Rempe \cite{rempe, rempe1, rempe2}, Rippon and Stallard \cite{rippon1, rippon2, rippon3, rippon4, rippon5}, Schleicher and  Zimmer \cite{sch} and  Sixsmith \cite{six1, six2}.  

An interesting topic is the Hausdorff dimension of the escaping set $I(f)$ investigated by McMullen in \cite{mac}. Later, Kotus and Bergweiler in \cite{KW}  showed that for a class of transcendental meromorphic functions the Hausdorff dimension of the escaping set is strictly less than 2. Other results concering the Hausdorff dimension  of Julia set  of transcendental meromorphic functions can be revised in \cite{stallard}.

In 2001, Baker, Dom\'{i}nguez and Herring \cite{bph} worked in a class of
meromorphic functions which contains the  class $\calk$. In this section we will present
some of  their results which apply to class $\calk$. 

We introduce the concept of  escaping set, folllowing \cite{bph},  for
functions in class $\calk$ as follows. Let $f\in \calk$ and $e \in B(f)$, the set

$$
I_e(f)=\{z\in \chat|  \lim_{n\rightarrow\infty}f^n(z)=e \}
$$

\noindent is {\it called the escaping set} to $e$  of $f$. 
A point $z\in I_e(f)$ is called an {\em escaping point} to $e$. The limit in the previous definition is well defined, so  $f^n(z)$ is well defined for all $n\in\N$. Thus, for all $n\in\N$, $f^n(z)\notin B(f)$. 

We refer to the following properties as the {\em  Basic Properties} of the escaping
set $I_e(f)$. We omit the proofs since they are similar to those in \cite{bph}.

\begin{theorem} 

Let $f\in \calk$ and $e\in B(f)$. Then:
\begin{enumerate}
\item $I_e(f)\neq \emptyset$,

\item $I_e(f)\cap J(f)\neq \emptyset$,

\item $\partial I_e(f)=J(f)$ and

\item $I_e(f)$ is completely invariant.
\end{enumerate}
\label{properties}
\end{theorem}

\begin{remark} (1) $I_e(f)\cap J(f)$ is completely invariant. (2) $I_e(f)$ is
  neither open nor close. (3) If $f\in \calk$ and $n\in \N$, then $I_e(f)\subseteq I_e(f^n)$.
\end{remark}

\begin{proposition}
Let $f\in \calk$ and $e\in B(f)$, then $I_e(f)\cap J(f)$ is a dense set in $J(f)$. 
\end{proposition}

\begin{proof}
Indeed by (1) of Theorem \ref{properties} there exist $z \in I_e(f)\cap J(f)$ since $\lim f^n (z) = e$, when  $n \to \infty $ and $f^n (z)  \neq  e$, then the cardinality of the forward orbit $O^+(z,f)$ is $\infty$. Therefore, there is  $w \in O^+(z,f) \subset J(f)$ such that $w$ is not a Picard exceptional value. Thus the backward   orbit  $O^-(w,f)$ is dense in $J(f)$.
\end{proof}

The following cases show examples for which $I_e(f)\cap F(f)\neq \emptyset$ .\\

(1) If $e\in B(f)$ is an absorbing point for an invariant Baker domain $U$, then
$U\subset I_e(f)$. So the escaping set of $f$ to a singularity $e$ may contain
points in the Fatou set,  as an example of this fact we have the function
$f(z) = e^{-z} + z  + 1$, see \cite{fatou}.\\

(2) Wandering domains may be  contained in the escaping set
for some $e\in B(f)$. The example is given by the function $f(z) = e^{-z} +
2\pi z  + 1$  in \cite{bergweiler1}.\\ 

We mention in the previous  section that functions $f \in S_K$ have neither Baker  nor   wandering domains, so we have  the following result.

\begin{theorem}
Let $f\in S_K$ and $e\in B(f)$, then $\overline{I_e(f)}=J(f)$.
\end{theorem}

\subsection{The general escaping set}



In order to get a better description on the dynamics of functions in class 
$\calk$ we introduce the general escaping set.

\begin{definition} \label{omegadef}
Let $f \in \calk$. The omega limit set of $z_0$ under $f$, denoted  by  $\omega (z_0,f)$, is defined as the set of  $z \in \chat$ such that there exist a subsequence  $\{f^{n_k}(z_0)\}$ contained in the forward  orbit $O^{+}(z_0, f)$ such that $lim f^{n_k}(z_0) = z$ as $k \to \infty$.
\end{definition}

\begin{remark}
If $\omega (z,f)$ is defined for a point $z$, then $z\notin B^-(f)$.
\end{remark}

Let $f\in \calk$, we define the {\it general escaping set} of $f$ as follows:

$$
I_g(f)=\{z\in\chat | \omega (z,f)\subseteq B^-(f)\}.
$$

Now we shall verify some basic properties of $I_g(f)$. 

\begin{theorem}
Let $f\in \calk$. Then

\begin{enumerate}
\item $I_g(f)\neq \emptyset$,

\item $I_g(f)\cap J(f)\neq \emptyset$,

\item $\partial I_g(f)=J(f)$ and 

\item $I_g(f)$ is completely invariant.

\end{enumerate}
\end{theorem}

\begin{proof}
For any $e\in B(f)$ we have $I_e(f)\subseteq I_g(f)$, so $I_g(f)\neq \emptyset$ and $I_g(f)\cap J(f)\neq\emptyset$. Now, points in $I_e(f)$ and repelling periodic points accumulate at any point in $J(f)$, therefore $\partial I_g(f)\supseteq J(f)$; points in the interior of $I_g(f)$ are in the Fatou set, so $\partial I_g(f)\subseteq J(f)$. Finally, $\omega (z,f)=\omega (f(z),f)$ and for any $w\in f^{-1}(z)$, $\omega (w,f)=\omega (z,f)$, thus $I_g(f)$ is completely invariant.
\end{proof}

\begin{proposition} \label{Iginvariante}
If $f\in \calk$ and $n\in \N$, then $I_g(f)\subseteq I_g(f^n)$.
\end{proposition}

\begin{proof} 
If $z\in I_g(f)$, then $\omega (z,f)\subseteq B^-(f)$. Since $\omega (z,f^n)\subseteq \omega (z,f)$ and $B^-(f)=B^-(f^n)$, then $z\in I_g(f^n)$.
\end{proof}

\begin{question}
When $I_g(f)= I_g(f^n)$?
\end{question}

For any periodic cycle of Baker domains $\{U_1, U_2, \dots U_n\}$
with absorbing points $z_1, z_2, \dots z_n$, so for all $z \in U_i$, $1 \leq i \leq n$, the omega limit $\omega (z,f) = \{z_1, z_2, \dots z_n \}\subseteq B^-(f)$ and therefore we have the following result.

\begin{proposition}
If $f\in \calk$ and $U$ is a periodic Baker domain, then $U \subset
I_g(f)$.
\end{proposition}


\begin{proposition}
If $f\in \calk$, then

\[
\bigcup_{e\in B(f)}I_e(f)\subseteq I_g(f).
\]
\end{proposition}

Indeed if $z\in I_e(f)$ for some $e\in B(f)$, then $\omega (z,f)=\{e\}$, thus $z\in I_g(f)$. 

For functions in class $\cale$, we have that $I_g(f)=I_{\infty}(f)$. For functions in class $\calm$ the usual escaping set concept is what we have called $I_{\infty}$. In contrast, the function $f(z)=-e^z+1/z\in\calm$ holds that $I_g(f)\neq I_{\infty}(f)$. Since $f$ has a cycle of period two of Baker domains, say $F_{\infty}$ and $F_0$, with $\infty$ and zero as the absorbing points respectively, then the Baker domains $F_{\infty}$ and $F_0$ are contained in $I_g(f)$ but not in $I_{\infty}(f)$, in fact $F_{\infty}\subset I_{\infty}(f^{2n})$ and $F_{\infty}\cap I_{\infty}(f^{2n+1})=\emptyset$, while $F_{\infty}\subset I_g(f^n)$ for all $n\in\N$.\\

\begin{proposition}
If $f\in \calk$ and $f$ has a cycle of Baker domains of period greater than one, with at least two different absorbing points for two Baker domains in this cycle, then

\[
\bigcup_{e\in B(f)}I_e(f)\varsubsetneq I_g(f).
\]
\end{proposition}

\begin{question} \label{pregoscilantes}
Is it true that if $f\in \calk$ and $\#B(f)>1$, then $\bigcup_{e\in B(f)}I_e(f)\varsubsetneq I_g(f)$?
\end{question}

As a consequence of some results in \cite{marti}, we know that for any function $f$ in $P_2$, there is a point $z\in I_g(f)$ such that $\omega (z,f)=\{0,\infty\}=B(f)$, therefore in this case the answer to Question \ref{pregoscilantes} is yes.


In a seminar, Nuria Fagella asked  the following question.

\begin{question}
If $z \in I_g(f)$ is it true that $\omega (f, z) \subset B(f)$? 
\end{question}

The answer is no because of the example $f(z)=-e^z + 1/z$ mentioned above. In this case all the
points in the Baker domains are oscillating between a singularity and a non
singularity point.

\begin{question} \label{pregpresingularidades}
For all $f \in \calk$ with at least one non
omitted singularity in $B(f)$, is it true  that $I_g(f)\neq\{ z \in I_g(f)|\omega (z, f) \subset B(f)\}$?
\end{question}

\begin{definition}
Let $f\in \calk$ and $E\subset B^-(f)$ such that the accumulation points of $E$ do not belong to $E$. For $I$ a countable set of indexes and let $E=\bigcup_{k\in I}\{e_k\}$. For all $e_k\in E$ choose an open neighbourhood $U_k$ of $e_k$ with the property that $U_k\cap U_j=\emptyset$ if $k\neq j$,  $s\in E^{\N}$. 

We say that  $s=\{s_1,...,s_n,...\}$ is \textit{eventually an itinerary} of $z\in I_g(f)$ if:

(1) $\omega(z,f)=\overline{E}$ and, 

(2) there exist $N,M\geq 0$ such that for all $n\in\N$, $f^{N+n}(z)\in U_k$ if and only if $s_{M+n}=e_k$.

In this case we say that $z$ escapes to $E$ with eventual itinerary $s$ and we say that $\{U_k\}_{k\in I}$ is a separating open cover of $E$.
\end{definition}

In the previous definition we choose $E$ to be finite or not. Observe that, by property (2), $s$ is eventually an itinerary of $z$ independently on how we choose the separating open cover of $E$. However if we change the separating open cover, the values $N$ and $M$ may change, so by the previous definition it is not hard to prove the following result.

\begin{proposition}
Let $f\in \calk$, $z\in I_g(f)$ and $s\in E^{\N}$ an eventual itinerary of $z$. $t\in E^{\N}$ is eventually an itinerary of $z$ if and only if the orbits of $s$ and $t$ under the shift map intersect.
\end{proposition}

\begin{definition} \label{clasificacionerrantes}
Let $z \in I_g(f)$.

\begin{enumerate}

\item $z$ is an {\it escaping oscillating} point of $f$ if 
the cardinality of its omega limit set is greater than one, this is  
 $\# \omega(z, f) > 1$,

\item $z$ is an {\it escaping periodic} point of $f$ if it has a periodic eventual itinerary,

\item $z$ is an {\it escaping wandering} point of $f$ if its omega limit set is an
infinite set.

\end{enumerate}
\end{definition}

If $z\in I_e(f)$ for some $e\in B(f)$, then $z$ is an escaping periodic point of $f$ of period 1. A point $z$ in a Baker domain $U$ is an escaping periodic point, where the period is $\# \omega (z,f)$ which is not always the period of $U$. For instance, all points in Baker domains of any entire function are escaping periodic points of period 1. In \cite{marti} is shown that all functions in $P_2$ have escaping oscillating points, some of them are periodic and others are non-periodic.

Using Definition \ref{clasificacionerrantes} we can reformulate the Questions \ref{pregoscilantes} and \ref{pregpresingularidades} as follows:

\begin{enumerate}
\item If $\# B(f)>1$, is it true that $f$ has escaping oscillating points?

\item Let $f\in \calk$ and $e\in B(f)$ such that is not an omitted value of $f$. Is there an escaping oscillating point $z\in I_g(f)$ such that $\omega (z,f)\nsubseteq B(f)$?
\end{enumerate}

\begin{question} Is there an example of a function in class $\calk$ with
  escaping wandering points?
\end{question}

\subsection{Escaping hairs}

For some transcendental meromorphic functions the escaping points can be connected to infinity with a curve $\gamma$  contained in $I_{\infty}(f)$, see for instance \cite{devaney}, \cite{rippon5} and \cite{RR}.

\begin{definition}
Let $f\in K$ and $e\in B(f)$. Suppose that there is a simple curve $\gamma\subset I_e(f)$ such that $e$ is an endpoint of $\gamma$, we called $\gamma$ a escaping hair to $e$ and $e$ is called a singular end point of $\gamma$.
\end{definition}

In the example $f(z)=-e^z+1/z$, the points in the Baker domain $F_0$ can be connected with a curve $\gamma$ to zero, which is in $B^-(f)-B(f)$ (not a singularity of $f$), $\gamma$ is contained in $I_g(f)\cap F_0$. With this in mind we state the following definition.

\begin{definition}
Let $f\in \calk$. Suppose that there exist $E\subset B^-(f)$, $\{U_k\}$ is a separating open cover of $E$, $s=\{s_1,...,s_n,...\}\in E^{\N}$ and a simple curve $\gamma\subset I_g(f)$ with end point $p\in B^-(f)$ such that:

\begin{enumerate}
\item for all $z\in \gamma$, $z$ escapes to $E$, 

\item exist $N,M\geq 0$ such that for all $n\in\N$, $f^{N+n}(\gamma)\subset U_k$ if and only if $s_{M+n}=e_k$.
\end{enumerate}

We say that $\gamma$ is an escaping hair to $E$ with eventual itinerary $s$
and a singular end point $p$.

\end{definition}

Observe that in the conditions of the previous definition, for all $z\in \gamma$, $s$ is an eventual itinerary of $z$.

\begin{definition}
Let $\gamma\subset I_g(f)$ an escaping hair to $E$ with eventual itinerary $s$ and singular end point $p_0$. We say that the set $O^*(p_0,\gamma,f)=\{p_0,p_1,...,p_n,...\}$ is a singular orbit of $p_0$, if $p_n\in B^-(f)$ is a singular end point of $f^n(\gamma)$ for all $n\in\N$.
\end{definition}

\begin{definition} \label{dinsingendpoint}
Let $O^*(p,\gamma,f)$ be a singular orbit of an endpoint $p$ of an escaping hair $\gamma$ to $E$, $p$ can be classified as follows:
\begin{enumerate}

\item periodic singular end point if a singular orbit is periodic,

\item pre-periodic singular end point if a singular orbit is pre-periodic,

\item oscillating singular end point if $\# E>1$ and,

\item wandering singular end point if its singular orbit is an infinite set.
\end{enumerate}
\end{definition}

Observe that (i) by definition every pre-periodic  singular end point is periodic and (ii) by  choosing different escaping hairs with a singular end point $p$, it  may have different classification.

\begin{question}
Is there an example of a function $f\in \calk$ and $p\in B^-(f)$ such that $p$ has singular orbits for each of the cases in the previous definition?
\end{question}



We are defining in some way the dynamics of a singular end point $p$ using escaping hairs, even when $f^n(p)$ is not defined for all $n\in\N$. In Section \ref{section6} we give an explicit example of a function in class $\calk$ with a wandering singular end point.\\

In the example $f(z)=-e^z+1/z$, there is $x\in \mathbb{R}^-$ such
that $\gamma=(-\infty,x]$ is an escaping hair to $E=\{\infty,0\}$ with singular
end point $\infty$ which is periodic, with singular orbit $O^*(\infty,\gamma,f)=\{\infty,0,\infty,0,...\}$. While $f(\gamma)$ is an escaping hair
with zero as a periodic singular end point and singular orbit $O^*(0,f(\gamma),f)=\{0,\infty,0,\infty,...\}$, see Figure \ref{marco}. Due to a result in \cite{RR}, $I_{\infty}(f)$ contains a escaping hair $\gamma$ with singular end point $\infty$, so its singular obit $O^*(\infty,\gamma,f)=\{\infty,\infty,\infty,...\}$. Therefore, in this case, $p=\infty$ has different dynamics according to Definition \ref{dinsingendpoint}.

\begin{figure}[h] 
\centerline{\includegraphics[height=7.5cm]{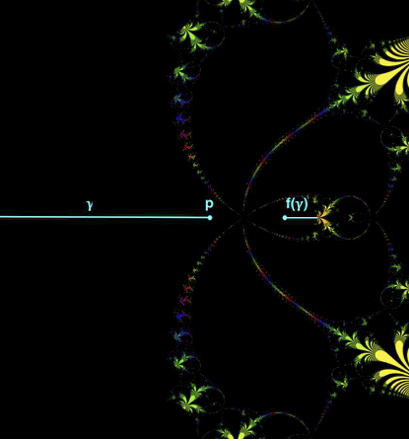}}
\caption{\small Escaping hairs with periodic singular end point.}
\label{marco}
\end{figure} 

We mention below some open question related to the previous discussion.

\begin{question}
Is there an example of a pre-periodic singular end point?
\end{question}

\begin{question}
Is there an example of an oscillating singular end point non periodic neither pre-periodic?
\end{question}

\begin{question} 
Is there an example of an escaping hair to $E$ such that $E$ is infinite?
\end{question}


 
 


\section{Example of escaping hairs with wandering singular end points}
\label{section6}


In this section we give two examples,  in the first  one   we show that for each of the finite essential singularities of the function  $g(z)=e^{\frac{1}{\sin z}}+z$, has infinitely many Baker domains attached to it. In the second one, for each finite essential singularity of  the function $f(z)=e^{\frac{1}{\sin z}}+z+2\pi$, there are infinitely many wandering domains attached to it. Our motivation to present these functions in class $\calk$, is to construct escaping hairs with wandering singular end points. 

\subsection{Properties of the functions $g$ and  $f$}

(1) The set $\{{\pi}k\}_{k \in {\Z}}$ is the set of the zeros of $\sin z$, thus  poles of $\frac{1}{\sin z}$, therefore essential singularities of $e^{\frac{1}{\sin z}}$ and so, for $g$ and $f$.

(2) To each essential singularity ${\pi}k$, $k\in\Z$, there are attached two tracts (due to the exponential function), one  is mapped to a punctured neighbourhood of $0$, and the other to a punctured neighbourhood of $\infty$. The points $0$ and $\infty$ are the only asymptotic singularities, see Figure \ref{tracts}.

(3)  These are tracts that intersect the real line, so we can talk of tracts at the right or left of the points ${\pi}k$, $k\in\Z$. Its dynamics is symmetric with respect to the real line. Moreover, tracts at the left of the points ${\pi}2k$ together with the tracts at the right of the points ${\pi}(2k+1)$ are mapped to a punctured neighbourhood of $0$. Also tracts at the right of the points ${\pi}2k$ together with the tracts at the left of the points ${\pi}(2k+1)$ are mapped to a punctured neighbourhood of $\infty$.\\

In the following section, we show that there are $g$-invariant domains contained in tracts of $g$ and they are the ones at the left of the points ${\pi}2k$, together with two tracts at $\infty$.

\begin{figure}[h]
\centerline{\includegraphics{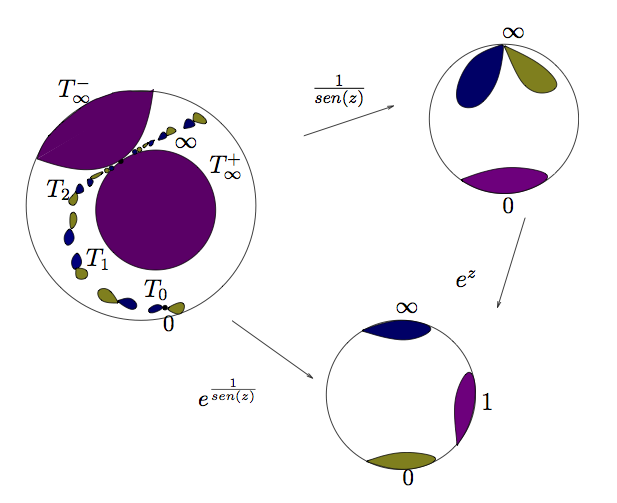}}
\caption{Tracts $T_{k}$ on yellow, $T^{+}_{\infty}$ and $T^{-}_{\infty}$ on  pink}
\label{tracts}
\end{figure} 

\subsection{Tracts containing invariant domains}

We will denote by $T_{k }$ the left tract  at $\pi 2k$ and by $T^{+}_{\infty }$ the upper tract at infinity. Since $g$ is symmetric respect the real line, there is also a lower tract at infinity, but the arguments are the same, so we omit details for this tract.

Consider $T^{+}_{\infty}=\{z: Img(z) > 3\}$ a logarithmic tract of $\sin z$ that is mapped to a pointed neighbourhood of $\infty$, by $1/\sin z$ to a pointed neighbourhood of $0$, which under the exponential function is mapped to a neighbourhood of $1$. Therefore $g(w)$ behaves as the function $w+1$, for $w \in  T^{+}_{\infty}$. It extends to an invariant Baker domain $B_{\infty}^+$ attached to infinity and containing 
$T^{+}_{\infty}$. Similarly, there is an invariant Baker domain attached to $\infty$, $B_{\infty}^-$. Clearly these Baker domains are simply parabolic.

\subsubsection{The tracts $T_{k}$}

In what follows we will give some general observations.\\

(1) Critical points of $g$ do not accumulate at the points ${\pi}2k$ near the real line, because the critical points satisfies:

\[
e^{\frac{1}{\sin z}}\left(-\cos z/\sin^2 z\right)+1=0,
\]

 that is, 

\[ 
 e^{\frac{1}{\sin z}}=\frac{\sin^{2}z}{\cos z}.
\] 
 
Neither of both functions oscillates near zero, since $\sin z\approx z$ at the points ${\pi}2k$, so $e^{1/z}$ does not oscillates (left hand of the equation) as well as $z^{2}$ (right hand of the equation).

(2) We identify $1/\sin z$ with $1/z$ in a very small neighbourhood of $0$. Let us consider $T_{k}$ a small disc at the left of ${\pi}2k$, so that it is sent to a left half plane by $1/z$ and to a small neighbourhood of $0$ by the exponential function. By the periodicity of $g$ it is enough to consider only $T_{0}$.

(i) If $T_{0}$ is bounded by the circle with center at $-{\pi}/8$ and radius ${\pi}/8$, $({\pi}/8)e^{it}-{\pi}/8$, then $(-{\pi}/4,0)=T_{0} \bigcap {\R}$. Thus $T_{0}$ goes under $1/z$ to the left half plane bounded by the vertical line with real coordinate $-\frac{1}{{\pi}/4}$.

(ii) Each circle, $C_{m}$, $\vert m\vert e^{it}+ im$, intersects $T_{0}$ in the curves $\sigma_{m}$, all tangent to the real line at $0$. Under $1/z$ , they go to $\frac{1}{me^{it}+im}=\frac{m\cos t}{m^{2}\cos^{2}t+m^{2}(\sin t+1)^{2}}-i\frac{1}{2m}$. In particular the curves $\sigma_{\frac{1}{4{\pi}k}}$ goes to half horizontal lines at hight $-2{\pi}ik$, see Figure \ref{new}. 

(iii) The function $g: T_{0} \rightarrow N_{0}$ sends each arc $\frac{1}{4{\pi}k}e^{it}+i\frac{1}{4{\pi}k}$ contained in $T_{0}$, to the real positive interval $(0, e^{-\frac{4}{\pi}})$, and the region in between two of such arcs to $N_{0}$, see Figure \ref{new}. 

(iv) Let $V_{0}$ the region contained in $T_{0}$ and is between the circles $C_{\frac{2}{\pi}}$ and $C_{-\frac{2}{\pi}}$, see Figure \ref{new}. 
 
\begin{figure}[h] 
\centerline{\includegraphics[height=9cm]{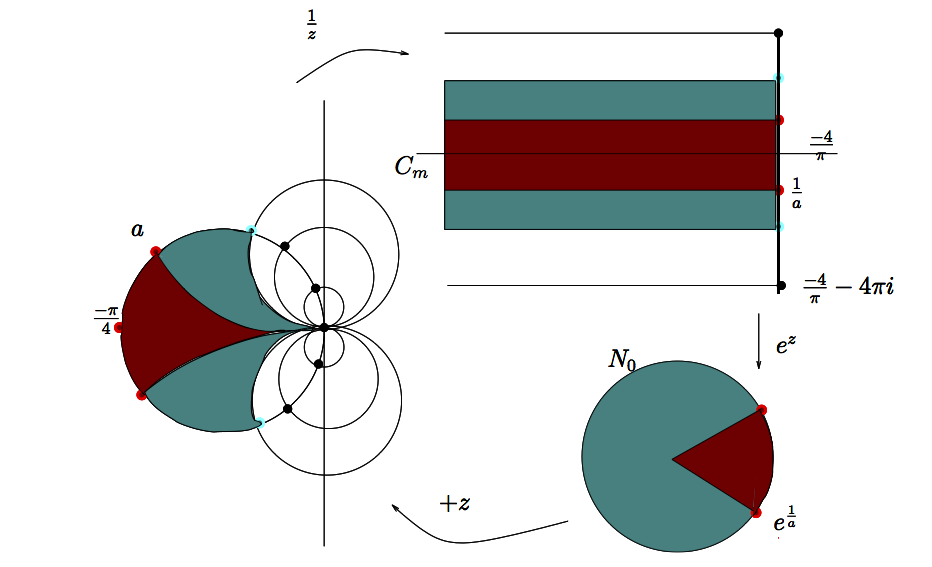}}
\caption{ On green  a fundamental region and  on red the corresponding absorbing domain $V_{0}$, $a$ as in the text.}
\label{new}
\end{figure}

From the previous observations, in the tract $T_0$ there are infinitely many regions $V_0^j$ for each $j\in\Z$. More over, for each singularity $\pi 2k$, there are infinitely many regions $V_k^j$ for $k$ and $j$ in $\Z$.

\begin{proposition} The regions $V_k^j$, $k,j\in \Z$, are absorbing domains of $g$.
\end{proposition}

\begin{proof}
Since $g$ is $2{\pi}$-periodic, it is enough to prove the existence of the absorbing domain $V_{0} \subset T_{0}$. For that, consider $V_{0}$, a neighbourhood of the real line in $T_{0}$ which is bounded by two curves ${\sigma}_{r}$ and ${\sigma}_{-r}$, with $r$ small enough, to be specified bellow.

We shall prove $g(V_{0}) \subset V_{0}$.

(i) We are interested in the horizontal distance, for that, consider the image of the interval $(-{\pi}/4,0)$ under $e^{1/z}$, which is the ray $(0, e^{\frac{-4}{{\pi}}}=(0, \frac{1}{e^{\frac{1}{{\pi}/4}}})$. Therefore if
$-t \in (-{\pi}/4,0)$, then $g(t)=-t+\frac{1}{e^{1/2t}}$, which is carried towards the right of $-t$ and $-t+\frac{1}{e^{1/2t}}<0$ for all $t>0$; now, by continuity this happens also for a small neighbourhood $V_{0}$ of the interval $(-{\pi}/4,0)$ in $T_{0}$. See Figure \ref{new}.

(ii) We are interested in the vertical distance, the goal is to find $r>0$ such that $V_{0}$ is bounded by ${\sigma}_{r}$ and ${\sigma}_{-r}$. For that, consider a point $w=-{\alpha}+iy \in V_{0}$,  above the real line, so $y>0$. Then, $\frac{1}{w}=\frac{-{\alpha}}{{\alpha}^{2}+y^{2}}+i{\frac{-y}{{\alpha}^{2}+y^{2}}}$ so $e^{1/w}=e^{(\frac{-{\alpha}}{{\alpha}^{2}+y^{2}})}e^{i({\frac{-y}{{\alpha}^{2}+y^{2}}})}$, with $\rho=\vert e^{1/w} \vert=e^{(\frac{-{\alpha}}{{\alpha}^{2}+y^{2}})}$, therefore $Img(g(w))=y+{\rho}\sin(\frac{-y}{{\alpha}^{2}+y^{2}})$.

We want $y+{\rho}\sin(\frac{-y}{{\alpha}^{2}+y^{2}})<y$. Since ${\rho}>0$, so we have to choose $\sin(\frac{-y}{{\alpha}^{2}+y^{2}})<0$, therefore we can choose $-{\pi}/4 \leq -\frac{y}{{\alpha}^{2}+y^{2}} \leq 0$.

Now, consider the piece of circle $\sigma_{r}$, where $y$ belongs, that is, ${\alpha}^{2}+(y-r)^{2}=r^{2}$, so ${\alpha}^{2}+y^{2}=2ry$. This implies that, $-{\pi}/4 \leq -\frac{y}{2yr} \leq 0$ i.e. $r \geq \frac{2}{\pi}$. To define $V_{0}$, let us choose $r=\frac{2}{\pi}$.

\end{proof}

The Baker domains attached to $0$, denoted by ${B}_{0}^j$, are those containing $V_0^j$ as absorbing domains. By the periodicity of the sine function, ${B}_{k}^j={B}_{0}^j+2{\pi}k$, $k \in {\Z}$.\\

As consequence of the discussion in this section, we have the following examples.

\begin{example}
The function $g(z)=e^{\frac{1}{\sin z}}+z$ contains infinitely many Baker domains $B_k^j$ of doubly parabolic type attached to each of the essential singularities $2{\pi} k$ of $g$.  $k\in\Z$. Also $g$ has a Baker domains $B_{\infty}^+$ and $B_{\infty}^-$ attached to $\infty$ of simple parabolic type.
\end{example}

\begin{example}
The function $f(z)=e^{\frac{1}{\sin z}}+z+2\pi$ has attached to each essential singularity $2 \pi k$ infinitely many wandering domains, $B_k^j$, see Figures \ref{guillermo} and \ref{guillermo2}.
\end{example}

\begin{figure}[h!]
\centerline{\includegraphics{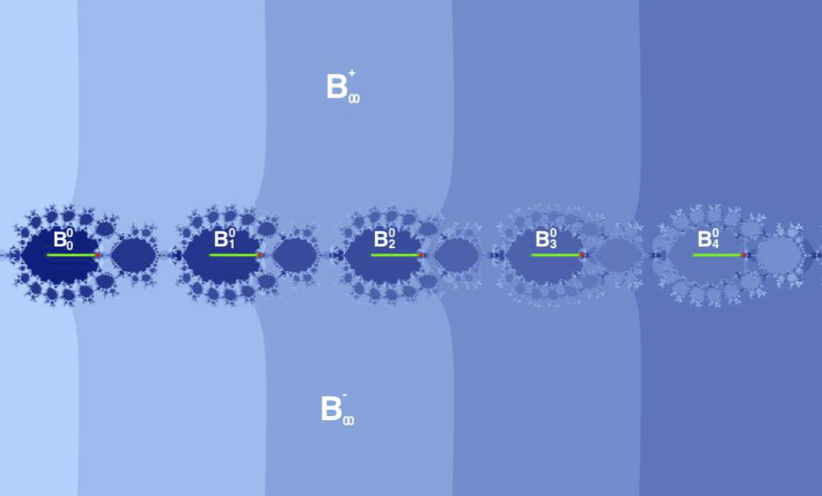}}
\caption{Escaping hairs with wandering singular end point on green and the singular end points on red in the dynamical plane.}
\label{guillermo}
\end{figure}

\begin{figure}[h!]
\centerline{\includegraphics[height=7cm]{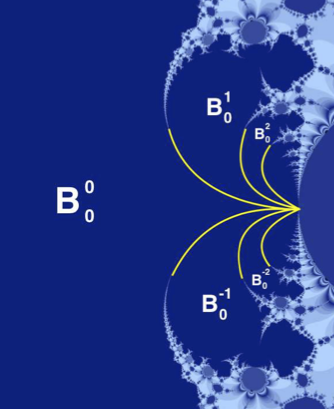}}
\caption{A detail of the dynamical plane of $f$ and $g$ showing the Fatou components $B_k^j$ close to zero. For $g$ these components are Baker domains and zero is their absorbing point, for $f$ they are wandering domains.}
\label{guillermo2}
\end{figure}

In what follows we give some problems  related to this section.
 
\begin{problem} Tracts at the right of ${\pi}(2k+1)$: In this case, we have that very near to the right of each singularity  ${\pi}(2k+1)$, $k \in {\Z}$; the function  $sin(z)$ is approximated by $-\frac{1}{z}$. Thus we have a tract  $g:T_{-\pi} \rightarrow N_{-\pi}$, with $T_{-\pi}$ a small disc at the right of $\pi$.

Have  the function $g(z)$ infinitely many Baker domains attached to each singularity ${\pi}(2k+1)$, $k \in {\Z}$, and whose intersection with the real line is empty?\\

\end{problem}

\begin{problem} Functions with periodic Baker domains: The function $z^{n}-1$ has critical points $0$ and $\infty$. Their zeros are the roots of unity, $w_{n}=e^{(\frac{2{\pi}i}{n})}$ and this points are the essential singularities of the maps $f_{n}(z)=e^{\frac{1}{z^{n}-1}}+(1/n)e^{(2{\pi}i/n)}z$.  

Has the function $f_{n}(z)$  a Baker domain of period $n$, each connected component attached to a different essential singularity?

\end{problem}

\subsection{Escaping hairs with wandering singular end points}

It is clear that each $B_k^j$ is contained in $I_g(g)$, with absorbing point $\pi 2k$, so they are contained in $I_{\pi 2k}(g)$, while  $B_{\infty}^{\pm}\subset I_{\infty}(g)$. According to Definition \ref{clasificacionerrantes}, we have that escaping points in all this Baker domains are not oscillating because they are escaping periodic of period 1.

In contrast, all wandering domains of $f$ are contained in $I_{\infty}(f)$. The interval $[-\pi/4,0)$ is contained in $B_0^0$, and by construction, it is an escaping hair with wandering singular end point $0$, with singular orbit $\{\pi 2k\}_{k\geq 0}$. 

\bigskip

\noindent{\bf{\it Acknowledgements.}} We wish to thank Professor I.N. Baker
because after 16 years that he had leaved this world he continues giving us
inspiration to continue working  on Complex Dynamics.

\end{document}